\documentclass[a4paper,11pt]{article}
\usepackage{amsmath,amssymb,amsfonts,amsthm,graphicx,fancyhdr}
\usepackage[latin1]{inputenc}
\usepackage[T1]{fontenc}
\usepackage{rotating,datetime,epic,multicol,tikz}
\usepackage{srcltx}
\usepackage[colorlinks=true]{hyperref}

\usetikzlibrary{calc}

\newdateformat{mydate}{\THEYEAR/\THEMONTH/\THEDAY}

\newtheoremstyle{drem}
     {3pt}
     {3pt}
     {\rmfamily}
     {}
     {\bf }
     {.}
     {.5em}
     {}

\theoremstyle{drem}

\newtheorem{teo}{Theorem}[section]
\newtheorem{lem}[teo]{Lemma}
\newtheorem{cor}[teo]{Corollary}
\newtheorem{ques}[teo]{Question}
\newtheorem{defi}[teo]{Definition}
\newtheorem{rem}[teo]{Remark}
\newtheorem{exa}[teo]{Example}

\newcommand{\eg}[0]{\emph{e.g.} }
\newcommand{\ie}[0]{\emph{i.e.} }

\newcommand{\pgen}[1]{\langle #1 \rangle}

\newcommand{\inj}[0]{\hookrightarrow}

\newcommand{\maxx}[1]{\textrm{\raisebox{.5ex}{\mbox{$\underset{#1}{\max}$}}} \,}
\newcommand{\minn}[1]{\textrm{\raisebox{.5ex}{\mbox{$\underset{#1}{\min}$}}} \,}

\newcommand{\rr}[0]{\ensuremath{\mathbb{R}}}
\newcommand{\zz}[0]{\ensuremath{\mathbb{Z}}}
\newcommand{\nn}[0]{\ensuremath{\mathbb{N}}}

\newcommand{\jo}[1]{\ensuremath{\mathcal{#1}}}

\newcommand{\del}[0]{\ensuremath{\partial}}

\DeclareSymbolFont{bbold}{U}{bbold}{m}{n}
\DeclareSymbolFontAlphabet{\mathbbold}{bbold}

\newcommand{\xvc}[1]{\overrightarrow{#1}}
\newcommand{\wh}[1]{\widehat{#1}}

\title{On Critical Nets in $\rr^k$}
\author{Antoine Gournay\thanks{Supported by the ERC Consolidator Grant No. 681207, ``Groups, Dynamics, and Approximation''.}~ and Yashar Memarian}

\begin{document}

\maketitle

\begin{abstract}
Critical nets in $\rr^k$ (sometimes called geodesic nets) are embedded graph with the property that their embedding is a critical point of the total (edge) length functional and under the constraint that certain $1$-valent vertices have a fixed position. In contrast to what happens on generic manifolds, we show that the total length of the edges not incident with a $1$-valent vertex is bounded by $rn$ (where $r$ is the outer radius), the degree of any vertex is bounded by $n$ and that the number of edges (and hence the number of vertices) is bounded by $n\ell$ where $n$ is the number of $1$-valent vertices and $\ell$ is related to the combinatorial diameter of the graph. 
\end{abstract}

\setcounter{teo}{0}
\renewcommand{\theteo}{\Alph{teo}}
\renewcommand{\theques}{\Alph{teo}}

\par A well-known problem in variational calculus is the Plateau problem: one looks for an $n$-dimensional object of minimal volume with prescribed boundary conditions. 
When $n=1$, this means one fixes a number of points and looks for a graph to join these points. Solutions are trees whose vertices have valency $3$ except for the boundary vertices (which are of valency $1$).
Furthermore, the edges must be geodesic segments and the tangent vectors of these geodesics meet at angle $2\pi/3$. 
However, this functional admits much more interesting critical points (sometime called geodesic nets).

The initial motivation for this note stems from \cite{Grom}. Roughly put, in this work, Gromov explores the space of chains and cycles of a manifold $M$ via the volume functional. Morse theory tells us the topology of the level sets change near critical points. The homology of the space of cycles being \emph{a priori} ``complicated'', there should be ``complex'' level sets and thus critical points. Bounds on the ``complexity'' of these critical cycles should allow one to get insight on the space of cycles. 

In this context it is natural to ask what kind of chains are not only minimal, but critical points of the length (1-dimensional volume) functional. 
Let $\jo{G}_n$ be the family of countable graphs with $n$ vertices of valency $1$ and no vertex of valency $2$. 
Then, for $1$-chains, this leads to the question of finding, among submersions of graphs ($\in \jo{G}_n$) in a manifold, those whose image admit a given boundary and are of critical length. 
These graphs will be called critical nets and the subject matter of this note is to study their properties. 
The bounds for ``complexity'' will be to bound the length of the interior edges, the valency of the vertices and the number of interior vertices in terms of the number of boundary vertices.
The bound on the vertices will be better expressed in terms of the combinatorial diameter of the graph.

These nets were also studied (in relation to minimal surfaces) by Markvorsen in \cite{Mar} where they bear the name of ``minimal webs''. 
Another motivation for this question comes from the work of Almgren.~In \cite[\S 4-4]{Alm}, Almgren shows that some stationary $1$-varifold (called ``spider web-like'') have infinitely many vertices. 
In \cite{AA}, Allard and Almgren show that if the densities of stationary $1$-varifold form a discrete set, then it is locally finite (see \cite[Theorem.(5) on p.85]{AA}). 
Stationary $1$-varifolds with constant density correspond to the critical nets considered here. 

\newlength{\tmpcomp}
\setlength{\tmpcomp}{\textwidth}
\addtolength{\tmpcomp}{-5cm}
\noindent
\parbox{\tmpcomp}{
However, this finiteness result may not be used to get bounds on neither the valency, nor the number of vertices. And for a good reason: on the sphere, there is no way to bound these quantities. Indeed, choosing $n$ generic geodesic circles creates a critical net with $n(n-1)$ interior vertices and no leaves (vertices of valency $1$), \eg
}
\parbox{5cm}{
\textcolor{white}{a} 
\hfill 
\includegraphics[height=3.5cm]{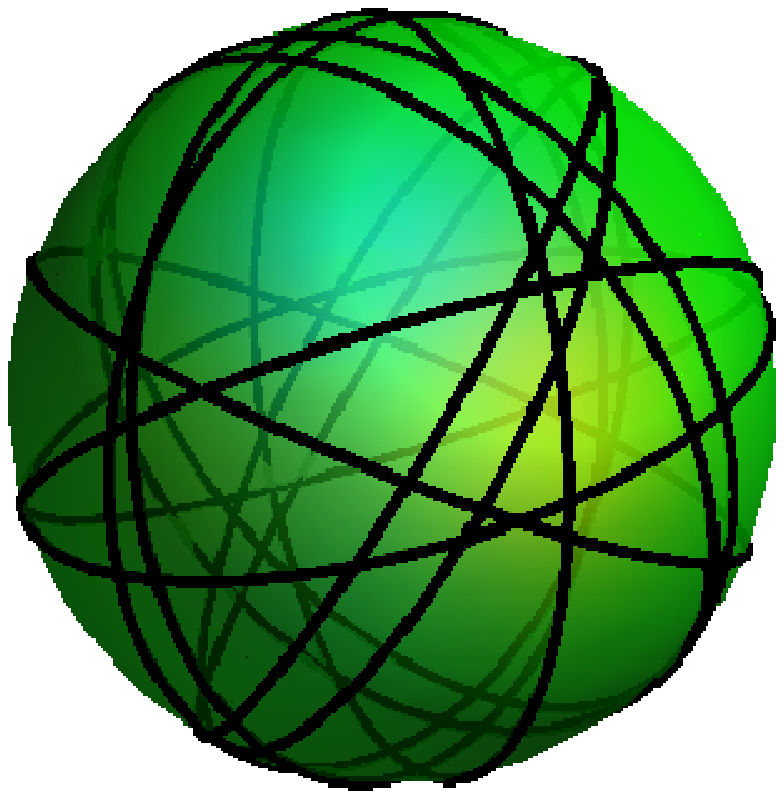}
\hfill 
\textcolor{white}{a} 
}

\noindent See the work of Hass and Morgan \cite{HM} for more on the possible constructions (there, the name is ``geodesic nets''). Similar constructions can be made on many other compact manifolds. 
Theorem \ref{lethm} shows the situation is quite different when we leave the world of positive curvature and compact spaces.

It is not so clear what a critical net with uncountably many vertices should be, or if such objects would be of any interest. As a consequence, $X$ (the set of vertices) will always be assumed countable. Also, it will be assumed that the critical net is contained in a bounded set.

In $\rr^k$, a critical net is described by the following data: there is a finite set of leaves, denoted $\del X$, of vertices whose position is given.
Henceforth, a vertex of valency $1$ will be called a leaf and the other vertices will be referred to as interior vertices. 

It is standard (see Almgren \& Allard \cite{AA} for a more general case) to check that conditions an embedded graph (which is a critical point of the length functional) will satisfy are that the edges must be geodesics.
Note it is still easy to see
\footnote{Assume the graph has a vertex of valency $>3$. Cut out a small ball around this vertex. Replace, inside this ball, the graph by a $3$-regular tree. 
Check that this operation reduces the length. 
Even if this operation creates a broken geodesic, it suffices to contradict minimality.}
that if one is interested in the minimal graphs, then the valency of any interior vertex is $3$. 
This particular case has been studied (for $\rr^2$) in \cite{SEXYashou}.

More recent works on the topic were made by Parsch (see \cite{P18} and \cite{P19}) as well as Nabutovsky and Parsch \cite{NP}. Note that the set-up there is slightly different (instead of leaves to attach the graph, they use vertex with possibly higher degree), but, as explained in of \cite{NP}, there is a correspondance between the two set-ups.


The main result here are the two following theorems. $X_{int}$ will denote the set of interior vertices and $\del X$ the set of leaves ($|\del X| \geq 2$). For a subset $S$ of vertices  $\nu(S) = \sum_{x \in S} \nu(x)$ where $\nu(x)$ is the valency of $x$. By the ``total length'' of a critical net, one should understand that the edges connected to a leaf are removed (otherwise it could be arbitrary large).

\begin{teo}\label{lethmnew}
 Let $G$ be a critical net in $\rr^k$ with finitely many leaves and contained in a bounded domain. Then the valency of an (isolated) vertex is bounded by $|\del X|$, the total length $L$ is bounded by $r |\del X|$ (where $r$ is the outer radius, \ie the radius of the smallest ball containing all inner vertices).
\end{teo}
The proof of Theorem \ref{lethmnew} is contained in \S{}\ref{sdef}. 
The bound on the degree remains true without the assumption on the isolated vertices, if one understands the degree as the length density near a vertex.
On the other hand, using length density, we also sketch an argument (not unlike that of Allard \& Almgren \cite{AA}) to show that the number of vertices is finite, see \S{}\ref{slenden}.

\begin{teo}\label{lethm}
Let $G$ be a critical net in $\rr^k$ with finitely many leaves and contained in a bounded domain. Then the total length $L$ satisfies
\[
L \leq \frac{ \big( \sum_{i=1}^k \ell_{v_i}^2  \big)^{1/2} }{2} |\del X|
\]
where $\ell_{v_i}$ are the lengths of the projections of the graph (with leaves removed) on orthogonal directions $\{v_i\}_{i=1}^k$. Furthermore, if $D_{v_i}$ is the length of the longest path by taking only edges which are positively colinear with $v_i$, then
\[
\nu(X_{int}) \leq 2 |\del X| + \big( \sum_{i=1}^k D_{v_i}^2  \big)^{1/2} |\del X|.
\]
\end{teo}
The proof of Theorem \ref{lethm} is contained in \S{}\ref{scurr}.

In a first approximation, $\tfrac{1}{2} ( \ell_{v_1}^2 + \ell_{v_2}^2  )^{1/2} $ may be thought of as the radius. The methods of the proof rely on looking at certain natural currents on the graph. In $\rr^2$, this can be geometrically seen as an isoperimetric property of packing of rectangles.

However, the estimates on the total length from Theorems \ref{lethmnew} and \ref{lethm} are rather sharp, as can be seen by looking at packings of regular hexagons, unit squares or unit (hyper)-cubes.
The estimate on $\nu(X_{int})$ is probably not sharp (see Question \ref{laques} below).

\begin{ques}\label{laques}
Let $G$ be a critical net in $\rr^k$. 
\begin{enumerate}\renewcommand{\labelenumi}{ \upshape \textbf{(\roman{enumi})}} \setlength{\itemsep}{0pt} 
\item Is there a constant $c$ such that $\nu(X_{int}) \leq c |\del X|^2$?
\item Fix the value of $|\del X|$ and look at all critical nets in $\rr^k$. If $G$ attains the maximal value of $|X_{int}|$ among all such graphs, is $G$ a packing of (not necessarily regular) hexagons, \ie $G$ is contained in a plane and $G$ is a $3$-regular graph?
\item Are there constants $K$ and $N$ so that $\displaystyle \big( \sum_{i=1}^k D_{v_i}^2  \big)^{1/2} \leq K |\del X|^N$?
\item Fix the value of $|\del X|$. Which $G$ attain the maximal value of $L/r$ where $r$ is the outer radius? Which $G$ attain the maximal value of $L/ \ell$, where $\ell = \big( \sum \ell_{v_i}^2 \big)^{1/2}$?
\hfill $\Diamond$
\end{enumerate}
\end{ques}
It would be nice to show that the value for $c$ in Question \ref{laques}{\bf .(i)} is $1/2$. Indeed, for a $3$-regular critical net in $\rr^2$ one would then obtain $|X_{int}| \leq \tfrac{1}{6} |\del X|^2$ and for a $4$-regular one, this yields $|X_{int}| \leq \tfrac{1}{8} |\del X|^2$. These inequalities coincide with the sharp inequalities done in \cite{SEXYashou} and would show that $3$-regular graphs are those with the highest number of interior vertices. 

Concering Question \ref{laques}{\bf .(iii)}, Example \ref{exadiam} shows that $N \geq 2$ and $K \geq \tfrac{1}{12}$. So even if a bound on the combinatorial diameter would be great (since it would yield an explicit bound on the number of vertices) it still cannot give a positive answer to Question \ref{laques}{\bf .(i)}

It seems reasonable to believe that the ``isoperimetric'' inequality obtained above extends to other spaces of non-positive curvature. For such spaces, the ``isoperimetry'' of these graphs  seems indifferent on the actual isoperimetry of the ambient space. For example, by looking at the intersection of $n$ bi-infinite geodesics in the hyperbolic plane, one may construct a critical net with $ n(n-1)/2$ interior points and $2n$ leaves. Lastly, by zooming in a sufficiently small neighbourhood, one expects to see the same behaviour as in $\rr^k$.

{\it Acknowledgments:} A. Gournay would like to thank A. Georgakopoulos for pointing Example \ref{exage} to him.

\section{Criticality, deformations and the ensuing properties}

\setcounter{teo}{0}
\renewcommand{\theteo}{\thesection.\arabic{teo}}
\renewcommand{\thelem}{\thesection.\arabic{teo}}


{\bf Notations.} 
For convenience, the set of edges will sometimes be seen as a (symmetric) subset of $X \times X$, $\hat{E}$, and somtimes as an unordered set of pairs $E$. 
In our set-up edges $(x,y)$ will be straight lines (geodesic) from $x$ to $y$ (in $\rr^k$ geodesics are unique so there is no ambiguity). 
As usual, the neighbours of $x$ are the set $N(x) = \{ y \in X \mid (x,y) \in \hat{E}\}$.

The {\bf leaf vector} refers to the unit vector pointing at a leaf of the graph. 

Sometimes vertices of degree two will be artificially introduced. 
In the context of deformations, they will be called ``anchor vertices''. 
In other context these will be called ``fake vertices''.

In order to be able to do anything on these graphs, it is necessary that the leaves are isolated from the other vertices.

Lastly, note that one can always create a new critical net out of an old one. 
Indeed, if $D \subset \rr^k$ be a (compact, smooth boundary) domain  
so that $\del D$ intersects $G$ only on its edges. 
The {\bfseries restriction} of $G$ to $D$, denoted $G_D$ is the graph constructed as follows: its vertices are the vertices of $G$ inside $D$ and its leaves are the anchors (intersection of the edges with $\del D$).

Note that, although we usually start with graph $G$ with finitely many leaves, it is possible for $G_D$ to have an infinite number of leaves.

\subsection{Deformations and defining criticality}

It is worthwhile to ponder a bit about the definition of critical nets. Sometimes, the definition is given as the necessary condition that at every vertex, the outgoing unit vectors of the edges add up to 0 (see Lemma \ref{tcondver-l}).
Though necessary this condition is not sufficient.
To this effect, A. Georgakopoulos pointed out the following example to the first author.

\begin{exa}\label{exage}

Consider an infinite 4-regular tree where the edges all meet at vertices with 90\textdegree (so all vertices look like a +).
There are bounded embeddings of this tree with finite total length.
Now cut some edge in two parts, remove one half of the tree and declare the point where you cut the tree a leaf. Here is what it could look like (only the first few levels of the tree are drawn)
\begin{center} 
%
%
%
\begin{tikzpicture}
\fill (-1,0) circle (.35ex);

\draw (-1,0) -- (0,0);

\draw (0,0) -- (.5,0);
\draw (0,0) -- (0,.5);
\draw (0,0) -- (0,-.5);

\draw (.5,0) -- (.5+.25,0);
\draw (.5,0) -- (.5,.25);
\draw (.5,0) -- (.5,-.25);

\draw (0,.5) -- (.25,.5);
\draw (0,.5) -- (-.25,.5);
\draw (0,.5) -- (0,.25+.5);

\draw (0,-.5) -- (.25,-.5);
\draw (0,-.5) -- (-.25,-.5);
\draw (0,-.5) -- (0,-.25-.5);

\newcounter{i}
\setcounter{i}{0}
\newcounter{j}
\setcounter{j}{0}
\foreach \point in {(.5+.25,0),(.5,.25),(.5,-.25),(.25,.5),(-.25,.5),(0,.25+.5),(.25,-.5),(-.25,-.5),(0,-.25-.5)}{
    \node[coordinate] (point-\arabic{i}) at \point { };
    \draw (point-\arabic{i}) -- ($(point-\arabic{i}) + (0,.125)$);
    \draw (point-\arabic{i}) -- ($(point-\arabic{i}) + (0,-.125)$);
    \draw (point-\arabic{i}) -- ($(point-\arabic{i}) + (.125,0)$);
    \draw (point-\arabic{i}) -- ($(point-\arabic{i}) + (-.125,0)$);
    \foreach \pointd in {(0,-1),(0,1),(-1,0),(1,0)}{
        \node[coordinate] (pointd-\arabic{j}) at \pointd { };
        \draw ($(point-\arabic{i}) + (0,.125)$) -- ($(point-\arabic{i}) + (0,.125) + .0625*(pointd-\arabic{j})$);  
        \draw ($(point-\arabic{i}) + (0,-.125)$) -- ($(point-\arabic{i}) + (0,-.125) + .0625*(pointd-\arabic{j})$);
        \draw ($(point-\arabic{i}) + (.125,0)$) -- ($(point-\arabic{i}) + (.125,0) + .0625*(pointd-\arabic{j})$);
        \draw ($(point-\arabic{i}) + (-.125,0)$) -- ($(point-\arabic{i}) + (-.125,0) + .0625*(pointd-\arabic{j})$);
        \stepcounter{j}
    }
    \stepcounter{i}
}



\end{tikzpicture}

\end{center}
This graph satisfies the condition of Lemma \ref{tcondver-l}, however it is not a critical net: you could move all vertices in such a way as to reduce the total length.
\end{exa}

Note that the condition that the crossing are at 90\textdegree{} is arbitrary and one could vary the angle at every crossing. One could also vary the length of the edges. If the graph is bounded the accumulation set of the vertices is going to be uncountable.

That said, there are essentially two ways of defining critical nets:

\begin{enumerate}
\item the graphs are countable and locally finite (meaning every vertex has finite degree, but the degree may not be bounded). This graph may be given the structure of a 1 dimensional pseudo-manifold (or look at it as a CW-complex). Look at embeddings $f:G \to \rr^k$ (inside a compact domain) and define the length of this embedding (possibly infinite). The embedding is critical if for any one parameter family of embedding which contain this embedding, the derivative of the length with respect to this parameter is 0. (Even if the length is infinite, one can study the variation of the length by restricting to parts whose length is finite.)
\item consider the (countable) graph as embedded in Euclidean space. 
It is critical if the variation of the length for any one-parameter family of diffeomorpisms (or even some particularly nice families of homeomorphisms) of Euclidean space is trivial. (Again, even if the total length of the edges is infinite, one may restrict to a part of the graph.)
\end{enumerate}

The second point of view is more restrictive: if one wants to use a deformation which moves a certain vertex $x$, then one might need to move vertices which are accumulating towards $x$. 
The second point of view is coherent with Allard \& Almgren \cite{AA};
in particular, note that two edges crossing have the same effect as one vertex of degree 4.
In the first point of view, there could be an embedding where two edges cross, and another one where they do not. 
The first seems more natural when one thinks of realisations of some homology class which are critical for the length functional.

Here is a more formal description of the second point of view. 
First, it is not too hard to see that for critical net in $\rr^k$, edges are given by straight segments. 
Otherwise, one could deform this edge in such a way that (even to the first-order) the variation of the lenght is non-zero.
This said the graph is then defined by knowing which vertices are neighbours and where they are. 

Given a vertex $x \in X$, the position in $\rr^k$ of this vertex will be denoted by $\vec{x}$.
The total length of a graph is given by:
\[
L = \sum_{ \{x,y\} \in E} \| \vec{x} - \vec{y}\|.
\]
Note that this quantity is \emph{a priori} not bounded.

Let $\phi_t$ be a family of Lipschitz homeomorphisms which is $\mathsf{C}^1$ with respect to $t$, so that $\phi_0$ is the identity and which is (for all $t$) the identity on all leaves of the graph. Using that $\|\vec{a} \| = \sqrt{ \vec{a} \cdot \vec{a}}$, The variation of the total length is
\[
\frac{\mathsf{d}}{\mathsf{d}t} \Big|_{t =0} \sum_{ \{x,y\} \in E} \| \phi_t(\vec{x}) - \phi_t(\vec{y})\|
= 
\sum_{ \{x,y\} \in E} \frac{\vec{x} - \vec{y}}{\| \vec{x} - \vec{y}\|} \cdot \big( \phi_0'(\vec{x}) - \phi'_0(\vec{y}) \big)
\]
(There are of course a certain number of conditions to clarify in order to make the above derivation rigourous.)

Note that $\phi'_0$ is the only important component of the right-hand side, and it can be chosen to be fairly arbitrary.
In order to keep things rigourous, 
we will restrict to the case where $\phi'_0$ is supported on the neighbourhood of a (compact with smooth boundary) domain $D$. 
Also it will be assumed that the boundary of the domain $D$ intersects only edges and that no vertex may accumulate to the boundary will be considered here.

Let us introduce some terminology in order to reduce the one-parameter family of deformations to a simpler setting.
At the intersection of every edge and $\del D$, one puts a ``{\bfseries anchor}'' vertex, a vertex of degree 2 which is not moved by the deformation.
An edge between an anchor vertex and a vertex inside $D$ will be called an 
{\bfseries anchor edge}.

If one includes these anchor vertices, then the deformation is completely determined by its effect on the vertices (the edges remain straight lines). 
A deformation is thus given by a function $X \to \rr^k$ which take value $\vec{0}$ on the anchors (since they are fixed) and is continuous (with respect to the topology on $X$ as a subset of $\rr^k$). 
This function $d$ says in which direction to perturb the vertices of $X$ and could be supported on a smaller set than $X_{int}$.

All this said the following lemma/definition summarises the only important property of critical nets that we will use.
\begin{defi}
 A graph is a {\bf critical net} if, for any continuous function $\vec{d}: X \to \rr^k$ whose support is a (compact for the topology of $\rr^k$) subset of $X_{int}$,
\[
\sum_{ \{x,y\} \in E} \frac{\vec{x} - \vec{y}}{\| \vec{x} - \vec{y}\|} \cdot \big( d(\vec{x}) - d(\vec{y}) \big) =0.
\]
where the sum runs over all unordered pairs $\{x,y\}$ so that at least one of $x$ and $y$ belongs to the support of $\vec{d}$. If only one of $\{x, y\}$ is in the sum, then the other one is an anchor vertex.
\end{defi}

Upon 
putting the anchor vertices on the other side, one gets the following:
\begin{lem}\label{tcritnet-l}
Assume $G \inj \rr^k$ is a critical net. 
For any $D \subset \rr^k$, for any continuous function $\vec{d}: X \to \rr^k$ whose support is in $D$,
\[
\sum_{\{x,y\} \subset D} \frac{\vec{x} - \vec{y} }{\| \vec{x} - \vec{y} \|} \cdot \big(\vec{d}(x) - \vec{d}(y) \big) 
= 
\sum_{ a \in A} \frac{\vec{a} - \vec{x}_a }{\| \vec{a} - \vec{x}_a \|} \cdot  \vec{d}( x_a )
\]
where the sum on the left runs over all unordered edges $\{x,y\}$ that lie in $D$ and the sum on the right runs over all the anchor vertices $A$ and $x_a$ is the vertex inside $D$ attached to the anchor $a$.
\end{lem}

Note that we do not require the total length to be finite.

\subsection{Basic properties, scaling and dilating }

This subsection just proves some basic properties and can be seen as some introductory examples for the application of Lemma \ref{tcritnet-l}. Note that the three lemmas of this subsection are already present in Parsch \cite{P18}, but the proofs follow a different method. 

\begin{lem}\label{tcondver-l}
In a critical net, for any [isolated] vertex $x$ (which is not a leaf) the following holds:
\[
 \sum_{y \in N(x)} \frac{\vec{x} - \vec{y} }{ \| \vec{x} - \vec{y} \|} =0
\]
\ie the sum of the unit vectors corresponding to the edges incident at $x$ sum up to 0.
\end{lem}
\begin{proof}
Using a function $\vec{d}$ which is supported exactly at the vertex $x$ one gets \linebreak $\displaystyle \sum_{y \in N(x)} \frac{\vec{x} - \vec{y} }{ \| \vec{x} - \vec{y} \|} \cdot \vec{d}(x) =0$. Since the value $\vec{d}(x)$ can be any vector of $\rr^k$, the claim follows.
\end{proof}

{\bfseries Scaling.} A particular deformation which will be useful is scaling the whole graph (except the leaves). 
For convenience it is better to let the centre of the scaling be the origin (so that $d(x) = \vec{x}$) and translate the whole graph so that the centre lies where one needs it.

\begin{lem}\label{ttwover-l}
If $G \inj \rr^k$ is a critical net, then $G$ has at least two leaves.
\end{lem}
\begin{proof}
Assume $G$ has only one leaf and pick $D$ to contain the whole graph but for a small neighbourhood of this leaf. 
There is only one anchor $a$;  let $x_a$ be its neighbour.
Pick $x_a$ as the centre of the scaling.
Then Lemma \ref{tcritnet-l} reads: $\sum \|\vec{x} - \vec{y} \| =0$.
So the graph is trivial.
\end{proof}

{\bfseries Dilating.} A similar deformation is a dilation. 
Here we will use it to mean that only the component toward some vector $\vec{e}$ will be scaled.
Again, for convenience it is better to let the centre of this transformation to be the origin (so that $d(x) = (\vec{e} \cdot \vec{x}) \vec{e}$ where $\vec{e}$ is some vector) and translate the whole graph so that the centre lies where one needs it.

\begin{lem}\label{tthreever-l}
If $G \inj \rr^k$ is a critical net with some interior vertex, then $G$ has at least three leaves.
\end{lem}
\begin{proof}
The previous lemma already show it cannot have one leaf. 
Assume $G$ has only two leaves and pick $D$ to contain the whole graph but for a small neighbourhood of these leaves. 
There are two anchors $a$ and $b$; let $x_a$ and $x_b$ be their neighbour.
Pick $x_a$ as the centre of the scaling and let $\vec{e}$ be any vector perpendicular to $\vec{x}_b - \vec{x}_a$ ($= \vec{x}_b$ since we translated the graph so that $\vec{x}_a =\vec{0}$).

Then Lemma \ref{tcritnet-l} reads: $\sum \frac{ ( \xvc{xy} \cdot \vec{e} )^2 }{\| \xvc{xy}\|} =0$ where $\xvc{xy} = \vec{y} - \vec{x}$.
Being a sum of positive numbers, each one must be 0. Since $\vec{e}$ can be chosen arbitrarily, this means all edges in the graph are collinear to $\vec{x}_a - \vec{x}_b$. 
In other words the graph is a line, and there are no interior vertices.
\end{proof}

\section{Using deformations}\label{sdef}

\subsection{Bounding the total length using scaling}

Recall that the outer radius of a subset $S$ of a metric space is infimum of all the $r$ so that there is a ball of radius $r$ covering $S$.

\begin{lem}\label{tbndlendef-l}
Let $G \inj \rr^k$ be a critical net, then 
$L = \sum_{x \in \del X} \xvc{N(x)} \cdot \hat{\ell}$ where $L$ is the total length (and $\xvc{N(x)}$ is the vector corresponding to the [unique] neighbour of a leaf).
\end{lem}
\begin{proof}
Look at a dilation of the whole graph except the leaves, that is $D$ excludes only the leaves and the deformation is $\vec{d}(x) = \vec{x}$ for any $x \in X_{int}$.

Then by Lemma \ref{tcritnet-l}
\[
\sum_{ \{x,y\} \in E} \| \vec{x}-\vec{y}\| 
=
\sum_{a \in A} \frac{\vec{a} - \vec{x}_a}{\| \vec{a}-\vec{x}_a\|} \cdot \vec{d}(x_a)
\]
The left-hand side is already the total length. As for the right-hand side note that it does not matter whether one consider the anchor $a$ or the leaf (which is basically just beside the anchor). Hence, if $\hat{\ell}$ denote the unit vector associated to a leaf edge
\[
L = \sum_{x \in \del X} \vec{x} \cdot \hat{\ell}\qedhere
\]
\end{proof}

\begin{cor}\label{tbndlen-c}
Let $r$ be the outer radius of a critical net whose the leaves (and their edges) have been removed. Then the total length is at most $r |\del X|$.
\end{cor}
\begin{proof}
Since the choice of the origin in Lemma \ref{tbndlendef-l} is arbitrary and $\displaystyle \sum_{x \in \del X} \xvc{N(x)} \cdot \hat{\ell} \leq \sum_{x \in \del X} \|\xvc{N(x)}\|$, the conclusion follows.
\end{proof}

Note that it is possible to find a bound on the total length by using dilation (instead of scaling). The bound is then the same as that of Lemma \ref{tisop-l}. Since the proof of Lemma \ref{tisop-l} is much more entertaining, we leave this alternative proof as an exercise.

Recall that we did not assume finiteness of the total length as an hypothesis of Lemma \ref{tbndlendef-l}.
\begin{cor}
Let $G \inj \rr^k$ be a (bounded) critical net (with finitely many leaves). Then its total length is finite.
\end{cor}

\subsection{Rotation and further use of scaling}

Let's call the unit vector of the edge at a leaf a leaf vector (stem vector would be botanically more pertinent, but mathematically confusing).
\begin{lem}\label{tzersumdef-l}
If $G$ is a (bounded) critical net in $\rr^k$, then the sum of the leaf vectors is $0$.
\end{lem}
\begin{proof}
Recall (see Lemma \ref{tbndlendef-l}) that scaling yields the equality:
$ \displaystyle L_{in} = \sum_{a \in A} \hat{\ell}_a \cdot \vec{x}_a$.
The choice of the origin does not play a r\^ole in the left-hand side, but does in the right-hand side. Let the origin move along a curved parametrised by $s$, then $ \displaystyle L_{in} = \sum_{a \in A} \hat{\ell}_a \cdot \vec{x}_a(s)$ (since $L_{in}$ and $\hat{\ell}_a$ do not depend on the position of the origin). Differentiating with respect to $s$ leads to:
\[
 0 = \sum_{a \in A} \hat{\ell}_a \cdot \vec{x}_a'(s)
\]
Note that all $x_a'(s)$ are equal (since we are only translating) and can be arbitrary. Hence $ \sum_{a \in A} \hat{\ell}_a=0$
\end{proof}
There is a physical interpretation of the above lemma. 
If all edges were some kind of springs which pull with unit length on both vertices at the end, then Lemma \ref{tbndlendef-l} just shows that the graph as a whole has zero net force.

See Lemma \ref{tconvhul-l} for an application of the previous lemma.

Note that Lemma \ref{tzersumdef-l} could also be proved by using the translation as a deformation ($d(\vec{x}) = \vec{e}$ for some fixed vector $\vec{e}$.) But differentating another deformation is a promising way of producing new quantitative equalities in critical nets.

In a similar vein, if one could look at a deformations given by rotating along the axis parallel to $\vec{e}$ through the origin: $d(x) = \vec{e} \times \vec{x}$. 
\begin{lem}\label{tzertordef-l}
If $G$ is a (bounded) critical net in $\rr^k$, then, for any $\vec{e}$,
$ \sum_{a \in A} \hat{\ell}_a \cdot (\vec{e} \times \vec{x}_a) =0$
\end{lem}
\begin{proof}
A direct application of Lemma \ref{tcritnet-l} to the above deformation yields:
\[
 0= \sum_{a \in A} \hat{\ell}_a \cdot (\vec{e} \times \vec{x}_a) \qedhere
\]
\end{proof}

Again there is physical interpretation of the previous lemma, namely: the torque on the whole graph is trivial.

\subsection{Bounding valency using swelling}

Let us introduce another particularly useful deformation which we will dub ``swelling'', namely $d(x) = \hat{x} := \frac{ \vec{x} }{\|\vec{x}\|}$ if $\vec{x} \neq \vec{0}$ and $= \vec{0}$ otherwise. 
Note that because of the continuity assumption of $d$, this deformation cannot be applied if there is an accumulation point at the origin.


\begin{lem}\label{tlemdeg-l}
Let $x$ be an isolated vertex, then $\nu(x) \leq |\del X|$.
\end{lem}
\begin{proof}
Take $D$ to be a domain which contains the whole graph (except the leaves) and consider the above deformation $d(x) = \hat{x}$.
Since it could happen that there is a vertex at the origin (the centre of the swelling), let us denote $0$ the vertex at the origin (if there is one), $\nu(0)$ the degree of that vertex and $E_0$ the set of unoriented edges without the edges incident with $0$.
With these notations, Lemma \ref{tcritnet-l} reads:
\[
\sum_{x \in N(0)} \frac{\vec{x} }{\| \vec{x} \|} \cdot  \hat{x}
+
\sum_{\{x,y\} \in E_0} \frac{\vec{x} - \vec{y} }{\| \vec{x} - \vec{y} \|} \cdot (\hat{x} -  \hat{y} )
= 
\sum_{ a \in A} \frac{\vec{a} - \vec{x}_a }{\| \vec{a} - \vec{x}_a \|} \cdot  \hat{x}_a
\]
The first term is $N(0)$, the second is positive and the sum on the right-hand side is bounded by $|\del X|$ (since there are as many anchors as leaves and the dot product of unit vectors is $\leq 1$). Hence $\nu(0) \leq |\del X|$
\end{proof}

\begin{rem}
There are two noticeable elements in the proof:
\begin{enumerate}
 \item If the origin goes through an edge, and one lets $E_0$ be the set containing all other edges, then the proof shows that 
 \[
\nu(0) 
+
\sum_{\{x,y\} \in E_0} \frac{\vec{x} - \vec{y} }{\| \vec{x} - \vec{y} \|} \cdot (\hat{x} -  \hat{y} )
= 
\sum_{ a \in A} \hat{\ell}_a \cdot  \hat{x}_a
\]  
where $\hat{\ell}_a$ is the unit vector of the anchor edge and $\nu(0)$ is the degree at the origin in the extended sense: $\nu(0)$ is the degree of the vertex at the origin if there is a vertex there, $\nu(0) = 2$ if an edge goes through the origin and $\nu(0) = 0$ if no part of the graph goes through the origin.

\item 
Let us call the optical length of the edge, the length of the arc if one centrally projects the edge on the unit circle.
It would be really nice if the quantity $\sum_{\{x,y\} \in E_0} \frac{\vec{x} - \vec{y} }{\| \vec{x} - \vec{y} \|} \cdot (\hat{x} -  \hat{y} )$ would be bounded from below in term of this optical length. 
But such a lower bound is not possible.
\end{enumerate}

\end{rem}

Lemma \ref{tlemdeg-l} and Corollary \ref{tbndlen-c} complete the proof of Theorem \ref{lethmnew}.

\subsection{Length density}\label{slenden}

In order to speak of length density, it is much more convenient to extend all leaves to infinity. 
That is, if $l$ is a leaf and $a$ its neighbour, then replace the segment from $a$ to $l$ by a half-line (or ray) starting at $a$ and going through $l$.
These half lines might intersect; just add an artificial vertex (of even degree) at those intersection points.

Now for any $r \in \rr_{>0}$, let $L(r)$ be the length of the graph inside the ball of radius $r$ around the origin
and let $\lambda(r) = \frac{L(r)}{r}$ be the length density.

It is easy to see that $\displaystyle \lim_{r \to \infty} \lambda(r) = |\del X|$ (since the total length, which excludes the leaf edges, is finite, see Corollary \ref{tbndlen-c}).
If there is no accumulation point at the origin, it is also clear that $\displaystyle \lim_{r \to 0} \lambda(r)$ tends to the degree of the vertex at the origin (degree 2 means that there is an edge through the origin).

\begin{lem}
$\lambda(r)$ is an increasing function. In fact, if $A_r$ denotes the intersection of the ball of radius $r$ with the graph,
\[
\lambda(r) = \sum_{a \in A_r} \hat{x}_a \cdot \hat{\ell}
\qquad \text{and} \qquad
\lambda'(r) = \frac{1}{r} \sum_{a \in A_r} \frac{1}{\hat{x}_a \cdot \hat{\ell}} - \hat{x}_a \cdot \hat{\ell}
\]
although $\lambda'$ may not be defined for countably many values of $r$ (at those values, the equality turns [in the sense of distributions] into a lower bound ). 
\end{lem}
\begin{proof}
First, note that if there is an edge tangent to the circle of radius $r$, then $\lambda$ will not be differentiable there. However, there are only countably many such edges (since there are countably many vertices). Thus these will be ignored for the rest of the proof. If one wishes to think as this differential in the sense of distributions, this makes the estimate a lower bound instead of an equality. 

Add some fake vertices (because they are of degree 2) at every point which intersect the sphere of radius $r$. 
(This is a way to make sure that the $\vec{x}_a$ all have the same norm; the set of anchors will be denoted $A_r$ instead of $A$ to stress this dependence.)
If there is a vertex on the sphere at that radius, just prolong all the edges coming from inside the sphere and put leaves a bit further.
Then by Lemma \ref{tbndlendef-l}
\[
L(r) = \sum_{a \in A} \vec{x}_a \cdot \hat{\ell} = r \sum_{a \in A} \hat{x}_a \cdot \hat{\ell}  
\qquad \text{hence } \lambda(r) = \sum_{a \in A_r} \hat{x}_a \cdot \hat{\ell}.
\]
We need to compute $\frac{\mathsf{d}}{\mathsf{d}r} \Big|_{r = \|\vec{x}_a\|} \hat{x}_a \cdot \hat{\ell}$.
Note that from the position $\vec{x}_a$ the edge is going in the direction $\hat{l}$. 
Hence, a first step is to compute
\[
\frac{\mathsf{d}}{\mathsf{d}t} \frac{\vec{y} + t \hat{\ell} }{\| \vec{y} + t \hat{\ell} \|} \cdot \hat{\ell} 
= 
\frac{\mathsf{d}}{\mathsf{d}t} \frac{\vec{y} \cdot  \hat{\ell} + t}{\| \vec{y} + t \hat{\ell} \|} 
= 
\frac{1}{\| \vec{y} + t \hat{\ell} \|} 
-
\frac{\vec{y} \cdot  \hat{\ell} + t}{\| \vec{y} + t \hat{\ell} \|^3  } 
( \vec{y} + t \hat{\ell} ) \cdot \hat{\ell} 
\]
Setting $t =0$, one gets
\[
\frac{\mathsf{d}}{\mathsf{d}t} \Big|_{t =0} \frac{\vec{y} + t \hat{\ell} }{\| \vec{y} + t \hat{\ell} \|} \cdot \hat{\ell} 
= 
\frac{1}{\| \vec{y}\|} 
-
\frac{(\vec{y} \cdot  \hat{\ell})^2 }{\| \vec{y}  \|^3  } 
=
\frac{1-(\hat{y} \cdot  \hat{\ell})^2 }{\| \vec{y}  \|  } 
\]
Next, one needs to differentiat with respect to $r$ and not $t$. Since $r = \| \vec{y} + t \hat{\ell} \|$, $\frac{\mathsf{d}r}{\mathsf{d}t}_{t =0} = \hat{y} \cdot \hat{\ell}$.
Hence,
$
\frac{\mathsf{d}}{\mathsf{d}r} \Big|_{r =\|\vec{y}\|} \hat{y} \cdot \hat{\ell} 
= 
(\frac{\mathsf{d}r}{\mathsf{d}t} )^{-1} \frac{\mathsf{d}}{\mathsf{d}t} \Big|_{t =0} \frac{\vec{y} + t \hat{\ell} }{\| \vec{y} + t \hat{\ell} \|} \cdot \hat{\ell} 
$.
This gives
\[
\frac{\mathsf{d}}{\mathsf{d}r} \Big|_{r =\|\vec{y}\|} \hat{y} \cdot \hat{\ell}
=
\frac{1}{\hat{y} \cdot  \hat{\ell}}
\frac{1-(\hat{y} \cdot  \hat{\ell})^2 }{\| \vec{y}  \|  } 
=
\frac{1}{\|\vec{y}\|} \; \Big( \frac{1}{\hat{y} \cdot \hat{\ell}} - \hat{y} \cdot \hat{\ell} \Big)
\]
Replacing $\vec{x}_a$ by $\vec{y}$ and using $\|\vec{x}_a\| = r$, one gets
\[ 
\lambda'(r) = \frac{1}{r} \sum_{a \in A_r} \frac{1}{\hat{x}_a \cdot \hat{\ell}} - \hat{x}_a \cdot \hat{\ell} 
\qedhere
\]
\end{proof}
Here is an important
\begin{cor}
The length density $\lambda(r)$ increases from the (generalised) degree to the number of leafs $|\del X|$.
\end{cor}

The results thus far may be used to prove the following 
\begin{lem}\label{tfinver-l}
Assume $G$ is a bounded critical net in $\rr^k$ with countably many vertices. If the number of leaves is finite, then the total number of vertices is finite.
\end{lem}
This lemma is a corollary of Almgren \& Allard's result \cite[Theorem.(5) on p.85]{AA}. However, 
let us indulge in a sketch of the proof (in $\rr^2$) without making appeal to this result. (The same argument applies for $\rr^k$.)

If there are infinitely many vertices (of degree $>2$), they must accumulate to some point $p$, since $G$ is bounded. 
By looking at the length density around $p$, one sees that only finitely many edges can point (up to some given deviation) radially toward $p$.
If there are infinitely many edges which are (up to some given deviation) tangent to circles around $p$, then these accumulate in some direction (as seen from $p$). 
Cutting along this direction through $p$ gives an unbounded current (see beginning of \S{}\ref{scurr}). 
And so the conclusion arises by contradiction.
Unfortunately, this (like \cite[Theorem.(5) on p.85]{AA}) is not a quantitative bound. There is still some work to turn this in a quantitative bound (like Question \ref{laques}{\bf .(i)}).


%
%
%

\subsection{Further relations}

One can also look at a family of dilations with respect to the vector $\vec{e}_s$ (with $s \in \rr$). Using Lemma \ref{tcritnet-l}, this yields
\[
\sum \frac{ ( \xvc{xy} \cdot \vec{e}_s )^2 }{\| \xvc{xy}\|} = \sum_{a \in A} (\hat{\ell}_a \cdot \vec{e}_s ) (\vec{x}_a \cdot \vec{e}_s) 
\]
Differentiating with respect to $s$ and letting $s=0$ gives:
\[
\sum 2 \frac{ ( \xvc{xy} \cdot \vec{e}_0 ) ( \xvc{xy} \cdot \vec{e}'_0 )} {\| \xvc{xy}\|} = \sum_{a \in A} (\hat{\ell}_a \cdot \vec{e}'_0 ) (\vec{x}_a \cdot \vec{e}_0) + (\hat{\ell}_a \cdot \vec{e}_0 ) (\vec{x}_a \cdot \vec{e}'_0) 
\]
Since $\vec{e}'_s$ is arbitrary (and writing $\vec{e}_0 = \vec{e}$), it follows that: for any $\vec{e}$,
\[
\sum 2 ( \xvc{xy} \cdot \vec{e} ) \wh{xy}  
= \sum_{a \in A} 
(\vec{x}_a \cdot \vec{e}) \hat{\ell}_a 
+ (\hat{\ell}_a \cdot \vec{e} ) \vec{x}_a
\]
It's not clear to the authors what is the interpretation of this equality, or how to put it to good use.

Similarly, in the definition of swelling the choice of the origin is arbitrary. 
By differentiating as in Lemma \ref{tzersumdef-l} one gets another new relation.

\section{Bounds using currents}\label{scurr}

\subsection{Currents}

One of our tools to obtain bounds are currents. These currents can be either defined as a consequence of Lemma \ref{tcritnet-l}, Lemma \ref{tcondver-l} or Lemma \ref{tzersumdef-l}.

{\bfseries Chopping.} 
Let us look at the ``chopping'' deformation (which is essentially a translation). In this deformation one takes some unit vector $v$ and pushes all vertices $x$ with $\vec{x} \cdot \vec{v} > r$ (for some $r \in \rr$) by $\vec{v}$.
Again, it is more convenient to take $r=0$ and translate the graph.
Note that because of the continuity assumption on $d$, this deformation cannot be applied if there is an accumulation point on the hyperplane (which is why it's good to introduce currents through Lemma \ref{tcondver-l} too).

Since $d(x)$ is either $\vec{v}$ or $\vec{0}$, Lemma \ref{tcritnet-l} shows that, 
\[
\sum_{ \{x,y\} \cap \vec{v}^\perp \neq \emptyset } \wh{xy} \cdot \vec{v} = \sum_{ a \in A} \wh{x_aa} \cdot \vec{v}
\]
where $\wh{xy} = \frac{ \xvc{xy}} {\| \xvc{xy} \|} =  \frac{ \vec{y} - \vec{x}}{\| \vec{y} - \vec{x} \|}$, 
the sum on the left is over all edges which cross the hyperplane $\vec{v}^\perp$ and the sum on the right is over all leaves (on the positive side of the hyperplane). Note that this is identical to applying Lemma \ref{tzersumdef-l} to a restricition of the graph.

It turns out this sum can be very conveniently interpreted as a current.

Let $v \in \rr^n$ be a unit vector. Define a current\footnote{It is a current as it will satisfy Kirchhoff's current law at each interior vertex; it does not, in general, satisfy Kirchhoff's potential law.} on the edges of the critical net as follows: if $(x,y)$ is an edge and $\wh{xy} = \frac{y-x}{\|y-x\|}$ is the associated unit vector, then the current from $x$ to $y$ is $c_v(x,y) := \wh{xy} \cdot v$. The current in the opposite direction has the opposite sign.

It is straightforward to check (using only Lemma \ref{tcondver-l}) that this is indeed a current: for $x$ an interior vertex, then $\sum_{y \in N(x)} \wh{xy} =0$, implies $\sum_{y \in N(x)} c_v(x,y) = \sum_{y \in N(x) } \wh{xy} \cdot v = 0$). As such, for every $v$, there are 3 types of leaves. Indeed, let $x$ be a leaf and $y$ its neighbour, then either\\
\noindent no current enters or leaves at $x$, \ie $\wh{yx} \cdot v =0$;\\
\noindent some current enters at $x$, \ie $\wh{yx} \cdot v > 0$;\\
\noindent some current leaves from $x$ \ie $\wh{yx} \cdot v< 0$.

Recall that the unit vector of the edge at a leaf is called a leaf vector.
The following lemma is already known, but the interest here is to offer a different proof of Lemma \ref{tzersumdef-l} (based on currents).
\begin{lem}\label{tzersum-l}
If $G$ is a (bounded) critical net in $\rr^k$, then the sum of the leaf vectors is $0$.
\end{lem}
\begin{proof}
Let $\{v_i\}_{1 \leq i \leq k}$ be an orthonormal basis. Since the current $c_{v_i}$ must enter and leave through the leaves, one has that $\sum_w w \cdot v_i =0$ where the sum runs on the leaf vectors. Since it holds for all $i$, $\sum w =0$.
\end{proof}

The following lemma is also already present in Parsch \cite{P18} (but again, the proof is based on a different principle).
\begin{lem}\label{tconvhul-l}
If $G$ is a (bounded) critical net, then $G$ is contained in the convex hull of its leaves
\end{lem}
\begin{proof}
If not, let $C$ be the convex hull of the vertices of $G$ and $P$ be the convex hull of its leaves. Let $H$ be a hyperplane separating some point $x \in C$ from $P$ and avoiding vertices of $G$. Let $D$ be the intersection of $C$ with the half-space bounded by the hyperplane $H$ and containing $x$. Then the leaf vectors of $G_D$ are all in the same (closed) half plane and at least one of them is in the open half-plane (otherwise $P$ is a line), so their sum may not be $0$.
\end{proof}

The upcoming lemma is identical to Lemma \ref{tthreever-l}, but the proof is based on currents (\emph{via} Lemma \ref{tconvhul-l}).
\begin{lem}
If $G \inj \rr^k$ is a critical net with some interior vertex, then $G$ has at least three leaves.
\end{lem}
\begin{proof}
If $G$ has two leaves, then the convex hull is a line, and there cannot be any vertex of degree $\geq 3$ on a line.
\end{proof}

\begin{lem}\label{tfaconv-l}
Let $G$ be a critical net in $\rr^2$. Then the closure of any bounded connected component of $\rr^2 \setminus G$ (\ie a face) is the convex hull of the vertices lying on its boundary, unless the face contains a leaf.
\end{lem}
\begin{proof}
Take $D$ to be a smaller and smaller neighbourhood of the face. Each vertex of the cycle must be attached to a leaf in $G_D$. Since $G_D$ is in the convex hull of its leaves (which tend to the vertices of the cycle), the conclusion follows.
\end{proof}

Many other nice properties can be shown using currents, see \S{}\ref{scurr}.

\subsection{Bounding valency}

The total entering current for the current $c_v$ (associated to $v$), noted $c_{in}(v)$, can be defined by $c_{in}(v) = \tfrac{1}{2} \sum_{x \in \del X, y \in N(x)} |c(x,y)|$ (recall that $N(x)$ contains only one element when $x \in \del X$). Denote the set of edges where currents leaves/enters the graph by $E^\pm_\del = \{ (x,y) \mid x \in \del X \text{ and } \pm c_v(y,x) >0 \text{ for } y \in N(x) \}$. One has the fairly intuitive description of $c_{in}$:
\begin{lem}\label{tcin-l}
$c_{in}(v) = \sum_{(x,y) \in E^+_\del} c_v(y,x) = \sum_{(x,y) \in E^-_\del} -c_v(y,x)$.
\end{lem}
\begin{proof}
Note that
\[
c_{in}(v) := \tfrac{1}{2} \sum_{x \in \del X, y \in N(x)} |c(x,y)| = \tfrac{1}{2} \sum_{(x,y) \in E^+_\del} c_v(y,x) + \tfrac{1}{2} \sum_{(x,y) \in E^-_\del} -c_v(y,x) .
\]
But by Lemma \ref{tzersum-l}, $\sum_{(x,y) \in E^+_\del} c_v(y,x) + \sum_{(x,y) \in E^-_\del} c_v(y,x) =0$. The conclusion follows. 
\end{proof}
This can be interpreted as follows: at each leaf, either some current enters or leaves the graph, and since there is as much current entering as leaving, the above sum indeed gives the quantity of current going through the graph. 
\begin{lem}
Let $G$ be a bounded critical net in $\rr^k$ and $\del X$ the set of leaves. Then $c_{in}(v) \leq |\del X| /2$. 
\end{lem}
\begin{proof}
The maximal current entering through an edge is $1$. Hence $c_{in}(v) \leq |\del X| /2$.
\end{proof}
The inequality is sharp (take a line with only two leaves). Except for this case (and disjoint unions of such), it is probably strict. Constructing graphs which approach this bound as nearly as desired is easy (look at nearly parallel lines).

For some non-zero $w \in \rr^k$, a $w^\perp$-cut of the net $G$ is a hyperplane parallel to $w^\perp$ which does not intersect the vertices of $G$. These cuts are useful in the following context: look at the current associated to $v$ and make a $v^\perp$-cut, \ie a hyperplane of the form $\lambda v + v^\perp$. Since it is assumed that $G$ has countably many vertices, there are (uncountably) many $v^\perp$-cuts. The current through the cut $\lambda v + v^\perp$ is defined as the sum of the currents on the edges which are positively colinear to $v$, \ie if $H_- = \cup_{\lambda'< \lambda} (\lambda' v+ v^\perp)$ and $H_+$ is the other half-space then
\[
c(\lambda v + v^\perp) := \sum_{x \in H_-, y \in H_+, x \in N(y)} c(x,y).
\]
Also, as $\lambda$ varies, this current only changes when the hyperplane passes through a leaf. Indeed, the current law implies the current through the cut does not change at an interior vertex. 
\begin{lem}
Let $G$ be a bounded critical net in $\rr^k$ with finitely many leaves. The current through a $v^\perp$-cut is, up to a sign, the sum (with sign) of the current entering or exiting in (either) one of the half-spaces determined by the $v^\perp$-cut. In particular, it is always less than $c_{in}(v) \leq |\del X|/2$
\end{lem}
\begin{proof}
Look at the graph restricted to $D$ a half-plane delimited by the $v^\perp$-cut. 
\end{proof}

Let us give an example of a nice consequence (though Lemma \ref{tlemdeg-l} already gives a better estimate). For this, recall that
\[
\kappa_{k}(n) = \max \bigg\{ \sum_{i=1}^n \|w_i\|_{\ell^1} \; \bigg| \; w_i \in \rr^k, \; \sum_{i=1}^n w_i = 0 \text{ and } \|w_i\|_{\ell^2}=1 \bigg\} \leq n \sqrt{k},
\]
since $\|w\|_{\ell^1} \leq \sqrt{k}\|w\|_{\ell^2}$. This is sharp for $n$ even.
\begin{lem}\label{tborndeg-l}
If $G$ is a bounded critical net in $\rr^k$, then the valency of any vertex is $\leq \sqrt{k} |\del X|$ where $|\del X|$ is the number of leaves.
\end{lem}
\begin{proof}
Consider a vertex $x$, then the sum of the $\ell^2$-norm of the unit vectors of the edges incident at $x$ is the valency $\nu(x)$ of $x$ (trivially). 

On the other hand, take $\{v_i\}_{i = 1, \ldots k}$ to be an orthonormal basis. For $i=1,\ldots, k$, the current $c_{v_i}$ along a $v_i^\perp$ cut before and after $x$ is at most $c_{in}(v_i)$. So the sum of the $\ell^1$-norm of those vectors is less than $\kappa_k(|\del X|)$. 

But $\|w\|_{\ell^2} \leq \|w\|_{\ell^1}$, so that the degree is less than $\kappa_k(|\del X|)$.
\end{proof}
Though the estimate of $\kappa_k(n)$ is sharp for $n$ even, the above bound on the valency is probably way off: the right bound is probably $|\del X|$.

When the set of vertices is discrete, it is possible to define the current $c_v$ through a vertex $x$ as the current $c_v$ through some hyperplane (perpendicular to $v$) close enough to $x$. Note that the proof of Lemma \ref{tborndeg-l} shows that the sum of the various currents through a vertex $x$ for an orthonormal basis $\{v_i\}$, (\ie $\sum_i c_{v_i}(x)$) is $\geq \nu(x) /2$ where $\nu(x)$ is the valency of $x$. In fact, since in $\rr^k$, $\|w\|_{\ell^1} \leq \sqrt{k} \|w\|_{\ell^2}$, one sees that the sum of currents through a vertex is between $\nu(x)/2$ and $\sqrt{k} \nu(x)/2$. 

\subsection{Bounding total length}

The bound on the total length is the one that comes most naturally from the currents.
\begin{lem}\label{tisop-l}
Let $G$ be a critical net in $\rr^k$ with finitely many leaves. Let $L$ be the total length of all edges (except the leaf edges), let $\ell_v$ be the size of the projection of the graph on an axis parallel to the vector $v$. Then for any orthonormal basis $\{v_i\}_{i=1,k}$,  
\[
\frac{2 L}{ \big( \sum_{i=1}^k \ell_{v_i}^2  \big)^{1/2} } \leq |\del X| 
\]
\end{lem}
We make first the proof for $k=2$ which has a nice interpretation in terms of rectangle packings.
\begin{proof}[Proof for $k=2$]
Assign a resistance $r_e$ to each $e$ given by the length of the edge $r_{(x,y)} = \| y-x\|$. This choice of resistance will make Kirchhoff's potential law to hold for any current $c_v$. Indeed, one only needs to check there is a potential function satisfying $p(y) - p(x) = r_{(x,y)} c_v(x,y)$. To do so, put $p(y) = y \cdot v$.

Stem edges can be assumed to be of length (hence resistance) equal to $1$. Realise the corresponding electrical current as a packing of rectangles (see \cite[\S{}II.2]{Boll}). Each rectangle corresponds to an edge. In the convention of \cite[\S{}II.2]{Boll}, its width is the current through the edge and its height the difference of potential at the extremities. In particular, the area of each rectangle is $r_e c_v(e)^2$. 

The total packing (without the leaf edges) has height the projection of the initial graph in the direction $v$ (\ie the biggest difference in potential), denote this by $\ell_v$, and width the total current associated to $v$, $c_{in}(v)$. Hence, one gets
\[
\sum_{e \in E_{int}} r_e c_v(e)^2 \leq c_{in}(v) \ell_v,
\]
where $E_{int}  = E \setminus E_\del$ is the set of edges not connected to leaves and $E_\del$ is the set of leaf edges.
Let $v_i$ be an orthogonal basis for $\rr^2$. Apply the previous inequality to $v_1$, then $v_2$ and sum. Note that $c_{v_1}(\cdot)^2 + c_{v_2}(\cdot)^2 = 1$, to obtain
\[
L = \sum_{e \in E_{int}} r_e \leq c_{in}(v_1) \ell_{v_1} + c_{in}(v_2) \ell_{v_2} \leq \big(  c_{in}(v_1)^2 + c_{in}(v_2)^2 \big)^{1/2} \big( \ell_{v_1}^2 + \ell_{v_2}^2  \big)^{1/2}
\]
where $L$ is the total length. To bound the right-hand side, recall that $2 c_{in}(v) \leq  \sum_{x \in \del X} c_v(x)$. Then, use the Cauchy-Schwarz inequality again:
\[
c_{in}(v)^2 \leq \tfrac{1}{4} \big( \sum_{x \in \del X} c_v(x) \big)^2 \leq \tfrac{1}{4} |\del X| \sum_{x \in \del X} c_v(x)^2. 
\]
Using again $c_{v_1}(\cdot)^2 + c_{v_2}(\cdot)^2 = 1$, one finds
\[
\frac{2 L}{ \big( \ell_{v_1}^2 + \ell_{v_2}^2  \big)^{1/2} } \leq |\del X|.\qedhere
\]
\end{proof}
Going back to $L \leq c_{in}(v_1) \ell_{v_1} + c_{in}(v_2) \ell_{v_2}$, it also possible to bound simply:
\[
\frac{L}{\max(\ell_{v_1},\ell_{v_2} ) } \leq c_{in}(v_1) + c_{in}(v_2) \leq \tfrac{\sqrt{2}}{2} |\del X|. 
\]
However, since $\sqrt{2} \max(\ell_{v_1},\ell_{v_2}) \geq \big( \ell_{v_1}^2 + \ell_{v_2}^2  \big)^{1/2}$, this bound is less interesting than the one in the lemma. They coincide when $\ell_{v_1} = \ell_{v_2}$.

\begin{proof}[Proof for $k \geq 2$]
For a function $f:X \to \rr$, let $\nabla f(x,y) = f(y) - f(x)$ be its gradient (seen as a function on the oriented edges). Then, using $\delta_\cdot$ to denote Dirac functions, 
\[
\pgen{ \nabla^* \delta_{(x,y)} , \delta_z} = \pgen{ \delta_{(x,y)} , \nabla \delta_z} = \left \lbrace \begin{array}{ll}
1 & \textrm{if } z=y \\
-1& \textrm{if } z=x \\
0 & \textrm{else}.
\end{array}
\right. 
\]
So the adjoint $\nabla^*$ of $\nabla$ (call it the divergence) is 
\[
\nabla^* g(x) =  \sum_{y \in N(x)} g(y,x) - \sum_{y \in N(x)} g(x,y) 
\]
Satisfying Kirchoff's current law is exactly being in $\ker \nabla^*$. Hence, for any current $c$ (as a function on the edges) and any $h: X \to \rr$ such that $h_{| \del X} \equiv 0$, 
\[
0 = \langle h \mid \nabla^*c \rangle = \langle \nabla h \mid c \rangle_{E_{int}} + \sum_{x \in \del X, y \in N(x)} h(y) c_v(x,y).
\]
Pick $h$ to be the potential from the previous proof and $c= c_v$. 
\[
\sum_{(x,y) \in E_{int}} \|y-x\| (\wh{yx}\cdot v)^2 = \sum_{x \in \del X, y \in N(x)} h(y) c_v(y,x).
\]
Split this last sum in two by considering $E^\pm_\del = \{ (x,y) \mid x \in \del X \text{ and } \pm c_v(y,x) >0 \text{ for } y \in N(x) \}$. Then
\[
\begin{array}{rl}
\displaystyle \sum_{x \in \del X, y \in N(x)} h(y) c_v(y,x) 
&= \displaystyle \sum_{(x,y) \in E^+} h(y) c_v(y,x) - \sum_{(x,y) \in E^-} h(y) c_v(x,y) \\
& \leq \displaystyle \maxx{(x,y) \in E^+_\del } h(y)  \sum_{(x,y) \in E^+} c_v(y,x) -  \minn{(x,y) \in E^-_\del} h(y) \sum_{(x,y) \in E^-} c_v(x,y)\\
& \leq \displaystyle \Big( \maxx{(x,y) \in E^+_\del } h(y)  -  \minn{(x,y) \in E^-_\del} h(y) \Big) c_{in}(v)\\
& = \ell_v c_{in}(v).
\end{array}
\]
The rest of the proof follows by the exact same computations as before.
\end{proof}


%
%

%
%

\subsection{Bounds on cardinality}

The next lemma is a combinatorial version of Lemma \ref{tisop-l}. To do so introduce the combinatorial length in a direction $v$, $D_v$, as follows. First, one requires that there are no edges $e$ such that $c_v(e) =0$ (this excludes finitely many $v$). Let $G'$ be the ``pruned'' critical net (\ie without the leaves). The edges of $G'$ by are given an orientation $\jo{O}_v$ in which, of the two oriented edges $(x,y)$ and $(y,x)$, only the one where $c_v$ is $> 0$ is kept. Then $D_v$ is the length of the longest oriented path in $G'$.
\begin{lem}\label{tiscomb-l}
Let $G$ be a bounded critical net in $\rr^k$ with finitely many leaves. Then 
\[
\nu(X_{int}) \leq 2|\del X| + \Big( \sum_{i=1}^k D_{v_i}^2 \Big)^{1/2} |\del X|.
\]
where $\{v_i\}_{i=1}^k$ is an orthonormal basis of $\rr^k$.
\end{lem}
\begin{proof}
Let $h$ be a function so that $\text{sgn} \nabla h(x,y) = \text{sgn} c_v(x,y)$ defined as follows. For a vertex $x$, $h(x)$ is the length of the \emph{longest} oriented path leading to $x$. Note that, if $c_v(x,y) >0$ then  $\nabla h(x,y) \in \zz_{>0}$ . As in the proof of Lemma \ref{tisop-l}, one has
\[
\sum_{e \in E_{int} } c_v(e)^2 \leq \sum_{e \in E_{int} } c_v(e) \leq  \langle \nabla h \mid c_v \rangle_{E_{int}} \leq D_v c_{in}(v),
\]
where, in the second sum, one chooses the orientation of the edge so that $c_v(e) \geq 0$.
Repeat and sum for each orthogonal direction to get 
\[
|E_{int}| \leq  \sum_{i} D_{v_i} c_{in}(v_i).
\]
The conclusion follows by noting that $|E_{int}| = \tfrac{1}{2}\nu(X_{int})- |\del X|$ and by bounding the right-hand term as in Lemma \ref{tisop-l}.
\end{proof}
Let us show an example where this inequality is sharp up to a constant.
\begin{exa}
Take the packing of the hypercube in $\rr^d$ of side $n \in \nn$ by unit hypercubes. Then $\nu(X_{int}) = 2d n^d$, $D_{v_i} = dn$ and $|\del X| = 2dn^{d-1}$. Thus the inequality reads 
$1 \leq \tfrac{2}{n} + d\sqrt{d}$. \hfill $\Diamond$
\end{exa}
On the other hand, There is an easy example to show that $D_v$ can be of the order of $|\del X|^2$. This shows that the previous lemma cannot give the bound stated in Question \ref{laques}.(i).
\begin{exa}\label{exadiam}
Take square side $n \in \nn$ packed by unit squares. Since $v$ cannot be parallel to the side, a longest path is a diagonal (of length $2n$). Now add $k$ lines to this graph, so that these line intersect every edge of this diagonal path. Then $D_v$ is $2(k+1)n$ while $|\del X| = 4n+2k$. Letting $k=n$ gives $D_v = \tfrac{1}{12} (|\del X|+6) |\del X|$. \hfill $\Diamond$
\end{exa}

Lemmas 
\ref{tisop-l}, \ref{tfinver-l} and \ref{tiscomb-l} complete the proof of Theorem \ref{lethm}.

\end{document}